\documentclass[11pt]{article}

\usepackage[a4paper,left=2cm,right=2cm,top=1.5cm,bottom=2cm]{geometry}
\usepackage[utf8]{inputenc}
\usepackage[T1]{fontenc}
\usepackage[english]{babel}
\usepackage{amsthm}
\usepackage{amsmath}
\usepackage{amssymb}
\usepackage{mathrsfs}
\usepackage{setspace}
\usepackage{float}
\usepackage{multirow}
\usepackage{stmaryrd}
\usepackage{color}
\usepackage{wrapfig}
\usepackage{pifont}
\usepackage{array}
\usepackage[colorlinks=true,linkcolor=ocre,citecolor=ocre]{hyperref}
\usepackage{dsfont}
\usepackage{upgreek}
\usepackage{nicefrac}
\usepackage{tikz}{}
\usepackage{gensymb}
\usepackage{FiraSans}
\usepackage{graphicx}
\usepackage{caption}
\renewenvironment{proof}{\noindent{\sffamily{\textbf{Proof :}}}}{\begin{flushright}$\square$\end{flushright}}

\newcommand{\IE}{\mathbb{E}}
\newcommand{\IN}{\mathbb{N}}
\newcommand{\IZ}{\mathbb{Z}}

\newcommand{\IR}{\mathbb{R}}

\newcommand{\IT}{\mathbb{T}}

\newcommand{\dst}{\displaystyle}
\newcommand{\drm}{\mathrm d}

\newcommand{\CD}{\mathcal D}
\newcommand{\CS}{\mathcal S}

\newcommand{\CE}{\mathcal E}

\newcommand{\CC}{\mathcal C}
\newcommand{\CH}{\mathcal H}

\newcommand{\CB}{\mathcal B}

\newcommand{\SF}{\mathscr{F}}

\newcommand{\IDC}{\mathds{1}}

\renewcommand{\P}{\mathsf{P}}

\newcommand{\PI}{\mathsf{\Pi}}

\newcommand{\eps}{\varepsilon}

\newcommand{\V}{V}
\newcommand{\af}{\mathbf{a}}

\definecolor{ocre}{RGB}{64,123,121}

\newcounter{item}
\numberwithin{item}{section}

\newtheorem{theorem}[item]{\sffamily Theorem}

\newtheorem{proposition}[item]{\sffamily Proposition}
\newtheorem{lemma}[item]{\sffamily Lemma}
\newtheorem{corollary}[item]{\sffamily Corollary}
\newtheorem{remark}[item]{\sffamily Remark}

\newtheorem*{theorem*}{\sffamily Theorem}
\newtheorem*{definition*}{\sffamily Definition}
\newtheorem*{proposition*}{\sffamily Proposition}
\newtheorem*{lemma*}{\sffamily Lemma}
\newtheorem*{corollary*}{\sffamily Corollary}
\newtheorem*{remark*}{\sffamily Remark}

\usepackage[explicit]{titlesec}
\titleformat{\section}{\centering\Large\bfseries}{\thesection \ --}{0.7em}{\Large\bfseries #1}
\titleformat{\subsection}{\centering\large\bfseries}{\thesubsection \ --}{0.4em}{\large\bfseries #1}
\titleformat{\subsubsection}{\centering\bfseries}{\thesubsubsection \ --}{0.4em}{\bfseries #1}

\providecommand{\MSC}[1]
{
	{\footnotesize	
	\textbf{MSC $\mathbf{2020}$ --} #1}
}

\providecommand{\keywords}[1]
{
	{\footnotesize	
	\textbf{Keywords --} #1}
}

\let\emph\relax
\DeclareTextFontCommand{\emph}{\bfseries\em}

\usepackage{mathtools}
\numberwithin{equation}{section}
\mathtoolsset{showonlyrefs}

\title{\bfseries Periodic nonlinear Schrödinger equation with distributional potential and invariant measures}
\author{Arnaud DEBUSSCHE and Antoine MOUZARD}
\date{}

\begin{document}

\maketitle
\abstract{In this paper, we continue some investigations on the periodic NLSE started by Lebowitz, Rose and Speer \cite{LRS} and Bourgain \cite{Bourgain94} with the addition of a distributional multiplicative potential. We prove that the equation is globally well-posed for a set of data $\varphi$ of full normalized Gibbs measure, after suitable $L^2$-truncation in the focusing case. The set and the measure are invariant under the flow. The main ingredients used are Strichartz estimates on periodic NLS with distributional potential to obtain local well-posedness for low regularity initial data.}
\vspace{0.5cm}

\MSC{35J10; 35Q55; 35A01; 35R60}


\section{Introduction}

Consider the nonlinear Schrödinger equation (NLSE) in the periodic setting
\begin{equation}\label{NLS}
i\partial_tu=-\partial_x^2u+\V u+\lambda|u|^{m-2}u
\end{equation}
where $u$ is a function on $\IT\times\IR$ with $\IT=\IR/\IZ$, $m\ge2$ an integer, $\lambda\in\IR$ and $\V$ a distributional potential. The difficulty of this equation lies in the roughness of the potential which makes the question of local well-posedness unclear. Indeed, the solution needs a priori to be regular enough for the product $\V u$ to make sense even for the linear equation
\begin{equation}\label{LS}
i\partial_tu=-\partial_x^2u+\V u
\end{equation}
with initial data $u(0)=u_0$. Our motivation comes from stochastic PDEs with $\V=\xi$ the spatial white noise on $\IT$, a random distribution in the Besov-Hölder space $C^{-\frac{1}{2}-\kappa}(\IT)$ for any $\kappa>0$. It corresponds to the derivative in the sense of distributions of a Brownian bridge and was first constructed by Paley and Zygmund \cite{PaleyZygmund30,PaleyZygmund32} as the random series
\begin{equation}
\xi(x)=\sum_{n\in\IZ}\xi_ke^{ikx}
\end{equation}
where $(\xi_k)_{k\ge0}$ is a family of independent and identically distributed random variables of centered standard complex Gaussian and $\xi_{-k}=\overline{\xi_k}$. In this work, we consider any potential $\V\in\CC^{-1+\kappa}(\IT)$ for any $\kappa\in(0,1)$. This corresponds to the Young regime which does not involve any renormalization procedure and includes the particular case of white noise. Other examples of such distributionnal potential can be random potential with different correlations or highly oscillatory potential such as
\begin{equation}
V(x)=\sum_{n\ge0}v_n\cos(nx)
\end{equation}
with a suitable growth condition on $(v_n)_{n\ge0}$ in order to ensure $V\in\CC^{-1+\kappa}(\IT)$. For the singular case beyond Young regime, see for exemple \cite{DRTV24,DW,TzvetkovVisciglia23,TzvetkovVisciglia23bis} with $\V$ a spatial white noise on $\IT^2$ or $\IR^2$ where the authors use an exponential transform or \cite{GUZ,Mouzard,MZ} which rely on a construction of the Anderson Hamiltonian using paracontrolled calculus. We follow the second approach to deal with low regularity initial data in a similar spirit to \cite{MZ}.

\smallskip

For such rough potential, the mild formulation
\begin{equation}
u(t)=e^{it\partial_x^2}u_0-i\int_0^te^{i(t-s)\partial_x^2}(Vu(s))\drm s
\end{equation}
is a priori doomed to fail for low regularity initial data due the singular product $Vu(s)$. This is in contrast with parabolic equation since the strong regularising properties of the heat semigroup are not available for the Schrödinger equation. For our purpose, we study the Hamiltonian
\begin{equation}
\CH=-\partial_x^2+\V
\end{equation}
which can be defined via its quadratic form $\af(u,v)=\langle\CH u,v\rangle$ with form domain $H^1(\IT)$. One can prove that there exists an associated operator with domain
\begin{equation}
\CD(\CH)=\big\{u\in H^1(\IT)\ ;\ \exists v\in L^2(\IT),\forall\varphi\in H^1(\IT),a(u,\varphi)=\langle v,\varphi\rangle\big\}
\end{equation}
and discrete spectrum $(\lambda_n)_{n\ge1}$ with an orthonormal basis of eigenfunctions $(e_n)_{n\ge1}$. Paracontrolled calculus gives us an explicit caracterization of the domain as
\begin{equation}
\CD(\CH)=\big\{u\in L^2(\IT)\ ;\ u-\P_uX\in H^2(\IT)\}
\end{equation}
where $X$ is a solution to $-\partial_x^2 X=\V$ and $\P$ denotes Bony's paraproduct constructed in \cite{Bony} using Paley-Littlewood decomposition, see Appendix \ref{SectionPaleyLittlewood}. In particular, the spectral calculus associated to $\CH$ yields solutions to associated linear equations such as \eqref{LS}.

\smallskip

A path to study the nonlinear equation \eqref{NLS} is to first understand the linear flow given by \eqref{LS}. The study of this equation without potential on a periodic domain goes back to Bourgain \cite{Bourgain94} motivated by Lebowitz, Rose and Speer \cite{LRS}. On the circle, the two major dispersive results on the free propagator are the Strichartz inequalities
\begin{equation}
\|e^{it\partial_x^2}u_0\|_{L^4([0,1]\times\IT)}\lesssim\|u_0\|_{L^2(\IT)}
\end{equation}
and
\begin{equation}
\|e^{it\partial_x^2}S_Nu_0\|_{L^6([0,1]\times\IT)}\lesssim N^\eps\|u_0\|_{L^2(\IT)}
\end{equation}
for any $\eps>0$ and $S_N$ the spectral projector on Fourier modes $|k|\le N$. This has to be understood as a trade of integrability between time and space. Indeed, the solution $u(t)=e^{it\Delta}u_0$ of the linear equation with initial data $u_0\in L^2(\IT)$ has conserved mass hence $u\in L^\infty(\IR,L^2(\IT))$ hence the trade of time integrability for spatial integrability, the second estimate being at the cost of an arbitrary small loss of positive Sobolev regularity for the initial data. This was used by Bourgain to prove local well-posedness for \eqref{NLS} for initial data in $H^\sigma$ with $\sigma\ge0$ in the cubic case $m=4$ and $\sigma>0$ in the case $4< m\le 6$. Following ideas from Burq, Gérard and Tzvetkov \cite{BGT} and Mouzard and Zachhuber \cite{MZ}, we obtain Strichartz inequalities for the linear propagator associated to $\CH$. This allows to use dispersive properties of the equation to get local well-posedness for low regularity initial data for equation \eqref{NLS}. We believe that Strichartz inequalities with such general deterministic rough potential are of independent interest. For example, this allows to generalize without loss the second estimate obtained by Bourgain for a potential $V\in L^\infty(\IT)$.

\smallskip

In order to extend local to global well-posedness, one usually relies on conserved quantity. For equation \eqref{NLS}, the two principal conserved quantities are the mass
\begin{equation}
\|u\|_{L^2(\IT)}^2=\int_\IT|u(x)|^2\drm x
\end{equation}
and the energy
\begin{equation}
\CE(u)=\int_\IT|\partial_xu(x)|^2\drm x+\int_\IT|u(x)|^2V(\drm x)+\lambda\int_\IT|u(x)|^m\drm x.
\end{equation}
The inital data needs to have finite energy in order to use its conservation, that is $u_0\in\CH^1(\IT)$. For low regularity initial data, the only conserved quantity is the mass hence local well-posedness in $L^2(\IT)$ yields global solutions, this was proved for the cubic equation $m=4$ without potential by Bourgain \cite{Bourgain93} while in the case $4<m\le 6$, he obtained local well-posedness in $H^\sigma(\IT)$ only for $\sigma>0$. In his following work \cite{Bourgain94}, he constructed global solution for random initial data distributed according to the Gibbs measure associated to the equation using the invariance of the measure instead of conserved quantity. We obtain a similar result here and construct global solutions for random initial data given by the Gibbs measure. We insist on the fact that low regularity solutions are challenging due to irregularity of the potential making it impossible to interpret equation \eqref{NLS} as a mild formulation associated to $-\partial_x^2$ even in the linear case. Our method for the globalisation differs from Bourgain \cite{Bourgain94} and seems new in the context of dispersive PDEs. It was first used in the parabolic case by Da Prato and Debussche \cite{DaPratoDebussche03}.

\smallskip


The Gibbs measure associated to equation \eqref{NLS} is formaly given by
\begin{equation}
\nu(\drm u)=\frac{1}{Z}e^{-\CE(u)}\prod_{x\in\IT}\drm u(x)
\end{equation}
where $\prod_{x\in\IT}\drm u(x)$ is the Lebesgue measure in infinite dimension and $\CE$ the energy. A first step in a rigorous definition of $\nu$ is to consider only the quadratic part with the Gaussian measure
\begin{equation}
\mu(\drm u)=\frac{1}{Z}e^{-\langle\CH u,u\rangle}\prod_{x\in\IT}\drm u(x)
\end{equation}
assuming for example that the operator is positive. Considering a basis of eigenfunctions with $(\lambda_n)_{n\ge1}$ and $(u_n)_{n\ge1}$ respectively the eigenvalues of $\CH$ and the coefficients of a function $u$ in the basis, the measure can be rigorously interpreted as the product measure
\begin{equation}
\mu(\drm u)=\prod_{n\ge1}\sqrt{\frac{\lambda_n}{\pi}}e^{-\lambda_n|u_n|^2}\drm u_n.
\end{equation}
We prove that this Gaussian measure associated to $\CH$ is supported in $C^{\frac{1}{2}-\kappa}$ for any $\kappa>0$, which is natural since the case $V=0$ corresponds to the law of the Brownian bridge. In particular, the potential energy $\int_{\IT^d}|u(x)|^m\drm x$ makes sense for $\mu$-almost all functions $u$. In the defocusing case $\lambda>0$, the Gibbs measure 
\begin{equation}
\nu(\drm u)=\frac{1}{Z}e^{-\lambda\int_{\IT^d}|u(x)|^m\drm x}\mu(\drm u)
\end{equation}
is well-defined as the density is bounded by one and $Z<\infty$. In the focusing case $\lambda<0$, this does not hold anymore and one has to consider a cut-off. Since the norm $\|u(t)\|_{L^2(\IT)}$ is conserved, a natural family of measure is given by
\begin{equation}
\nu_B(\drm u)=\frac{\IDC_{\|u\|_{L^2}\le B}}{Z_B}e^{-\lambda\int_{\IT^d}|u(x)|^m\drm x}\mu(\drm u)
\end{equation}
for any $B>0$. In the case $4\le m<6$, we prove that this measure is invariant for generic cut-off parameter $B$ while a smallness assumption on $B$ is needed in the case $m=6$. The introduction of a mass cut-off goes back to Lebowitz, Rose and Speer \cite{LRS}, see also \cite{OhSosoeTolomeo22} for a recent result for the critical parameter $B$ in the case $m=6$ and Remark \ref{RemarkCriticalCase} for the relation with our measure.

\smallskip

In Section \ref{SectionOperator}, we study the Hamiltonian $\CH$ and give the first properties of its associated Schrödinger group $e^{it\CH}$ using paracontrolled calculus. In Section \ref{SectionStrichartz}, we prove Strichartz inequalities which yields local well-posedness for low regularity initial data for equation \eqref{NLS}. We also consider the truncated version of the equation with the spectral projector associated to $\CH$ and prove the convergence of the solutions to the untruncated equation. In Section \ref{SectionMeasure}, we construct the associated Gibbs measure with a suitable cut-off in the focusing case and prove global well-posedness on its support with invariance of the measure for $m\le 6$. The Paley-Littlewood decomposition and the paraproduct are presented in Appendix \ref{SectionPaleyLittlewood}.

\medskip

{\bf Acknoledgments :} The first author benefits from the support of the French government “Investissements d’Avenir” program integrated to France 2030, bearing the following reference ANR-11-LABX-0020-01 and is partially funded by the ANR project ADA. The second author is grateful to Tristan Robert, Hugo Eulry and Nicolas Camps for interesting discussions about this work.

\section{Schrödinger operator with distributional potential}\label{SectionOperator}

In this section, we study the Hamiltonian
\begin{equation}
\CH=-\partial_x^2+V
\end{equation}
with $V\in\CC^{-1+\kappa}(\IT)$ for $\kappa\in(0,1)$. Due to the roughness of the potential, the domain of the operator does not contain smooth functions since $\CH u\in\CC^{-1+\kappa}(\IT)$ for $u\in C^\infty(\IT)$. On the other hand, the associated quadratic form
\begin{equation}
\af(u,v)=\langle\CH u,v\rangle
\end{equation}
is well-defined for any $u,v\in H^1(\IT)$, this is the content of the following proposition. In the following, we do not keep the space $\IT$ in the notation since we always work on the circle.

\begin{proposition}
The form $(\af,H^1)$ is a closed continuous symmetric form. It is quasi-coercive, that is there exists a constant $c>0$ such that
\begin{equation}
\af(u,u)+c\|u\|_{L^2}^2\ge\frac{1}{2}\langle\partial_x u,\partial_x u\rangle
\end{equation}
for any $u\in\CH^1$.
\end{proposition}

\begin{proof}
We have
\begin{equation}
|\langle \V u,u\rangle|\lesssim\|u^2\|_{\CB_{1,1}^{1-\kappa}}\|\V\|_{\CC^{-1+\kappa}}\lesssim\|u\|_{H^{1-\kappa}}^2\|\V\|_{\CC^{-1+\kappa}}
\end{equation}
using ii) and iii) from Proposition \ref{PropBesov}. Then for any $\eps>0$, there exists a constant $c_\eps>0$ such that
\begin{equation}
\|u\|_{H^{1-\kappa}}^2\le\eps\|u\|_{H^1}^2+c_\eps\|u\|_{L^2}^2
\end{equation}
from standard interpolation inequality hence there exists a constant $c>0$ such that
\begin{equation}
|\langle \V u,u\rangle|\le\frac{1}{2}\|\partial_xu\|_{L^2}^2+c\|u\|_{L^2}^2. 
\end{equation}
We get
\begin{align}
\langle(\CH+c)u,u\rangle&=\langle \partial_xu,\partial_xu\rangle+\langle \V u,u\rangle+c\langle u,u\rangle\\
&\ge\frac{1}{2}\langle\partial_x u,\partial_x u\rangle
\end{align}
hence the quasi-coercive property. The symmetry of the form directly follows from an integration by part and that the potential is real, it only remains to prove closedness. Let $(u_n)_n\subset H^1$ such that
\begin{equation}
\lim_{n\to\infty}\|u-u_n\|_{L^2}+\af(u-u_n,u-u_n)=0
\end{equation}
for some $u\in L^2$. The previous bounds gives that $(\partial_xu_n)_n$ is a Cauchy-sequence in $L^2$ thus converges to a limit, that is $\partial_xu$ hence $u\in H^1$.
\end{proof}

\begin{remark}
Considering the derivative of a Brownian bridge $\V=\drm B\in C^{-\frac{1}{2}-\eps}$ for any $\eps>0$ corresponds to the Anderson Hamiltonian, this was its first construction by Fukushima and Nakao \cite{FN} on a finite segment. See also \cite{MatsudaZuijlen22,MouzardOuhabaz23} for a construction of the form in the singular case.
\end{remark}

It follows from the previous proposition that there exists a self-adjoint operator $\CH$ with dense domain 
\begin{equation}
\CD(\CH)=\big\{u\in L^2(\IT)\ ;\ \exists v\in L^2,\forall\varphi\in H^1,\af(u,\varphi)=\langle v,\varphi\rangle\big\}\subset H^1,
\end{equation}
see for example Ouhabaz's book \cite{Ouhabaz05}. Moreover, the operator is self-adjoint, densely defined and bounded from below. Since $H^1$ is compactly imbedded in $L^2$, $\CH$ has discrete spectrum $\lambda_1\le\lambda_2\le\ldots$ with an associated basis $(e_n)_{n\ge1}$ of $L^2$. Note that since the form domain of $\CH$ and of the Laplacian are the same, the first eigenvalue is simple, that is $\lambda_1<\lambda_2$ and there exists a positive ground state $\Psi\in\CD(\CH)$, see \cite[Theorem 4.1]{MouzardOuhabaz23} for the details. The flow of the Schrödinger linear equation associated to $\CH$ has the following spectral representation
\begin{equation}
e^{it\CH}u_0=\sum_{n\ge1}e^{it\lambda_n}\langle u_0,e_n\rangle e_n
\end{equation}
for any $t\in\IR$ and $u_0\in L^2$. This yields a weak solution to the linear equation
\begin{equation}
i\partial_tu=-\partial_x^2u+\V u
\end{equation}
with initial data $u(0)=u_0$, the equation being interpreted in the dual of the domain $\CD(\CH)^*$. This is coherent with the case $V=0$ where the linear propagation of any $u_0\in L^2$ solves the equation in $H^{-2}=(H^2)^*$. However, the main difference here is that the domain $\CD(\CH)$ does not contain smooth functions hence its dual $\CD(\CH)^*$ is not a subspace of distributions and the equation is not satisfied in the sense of distributions. For the nonlinear equation \eqref{NLS}, we will need a finer description of the operator as well as its domain. 

\smallskip

In the following, we use the paracontrolled calculus to construct a map $\Gamma:L^2\to L^2$ depending on $\V$ such that the operator
\begin{equation}
\CH^\sharp=\Gamma^{-1}\CH\Gamma
\end{equation} 
is a better behaved perturbation of the Laplacian than $\CH$. In particular, for any $u\in C^\infty$, we have
\begin{equation}
(\CH+\partial_x^2)u=Vu\in C^{-1+\kappa}
\end{equation}
while it will be crucial that
\begin{equation}
(\CH^\sharp+\partial_x^2)u\in C^{2\kappa}\subset L^2.
\end{equation}
However the map $\Gamma$ is not a unitary transformation of $L^2$ thus $\CH^\sharp$ is not longer self-adjoint. The idea of such transformation was first used by Gubinelli, Ugurcan and Zachhuber \cite{GUZ} and Mouzard \cite{Mouzard} to study the Anderson Hamiltonian in two dimensions. This was crucially used to obtain Strichartz inequalities by Mouzard and Zachhuber \cite{MZ} as well as precise small time asymptotic and two-sided Gaussian bounds for the heat semigroup $e^{-t\CH}$ by Bailleul, Dang and Mouzard \cite{BailleulDangMouzard22}. We refer to Appendix \ref{SectionPaleyLittlewood} for the definitions of the tools from paracontrolled calculus such as the paraproduct $\P$ and the resonant product $\PI$.

\smallskip

For $u\in\CD(\CH)$, we have $\CH u\in L^2$ hence
\begin{align*}
\Delta u&=\V u-\CH u\\
&=\P_u\V+\P_\V u+\PI(u,\V)-\CH u
\end{align*}
with $\Delta=\partial_x^2$ hence
\begin{align*}
u&=\Delta^{-1}\P_u\V+\Delta^{-1}\big(\P_\V u+\PI(u,\V)-\CH u\big)\\
&=\P_uX+[\Delta^{-1},\P_u]\V+\Delta^{-1}\big(\P_\V u+\PI(u,\V)-\CH u\big)
\end{align*}
with $X=\Delta^{-1}\V\in\CC^{1+\kappa}$. We denote here as $\Delta^{-1}$ the inverse of the Laplacien defined on centered function on $\IT$ hence the rigorous definition of $X$ is
\begin{equation*}
X:=\Delta^{-1}(\V-\langle\V,1\rangle).
\end{equation*}
Since $u\in\CD(\CH)\subset H^1$, the roughest term is given by $X$ and this yields the ansatz
\begin{equation}
\CD_\V:=\{u\in L^2\ ;\ u-\P_uX\in H^2\}
\end{equation}
for the domain. Following the previous computation, one has $\CH u\in L^2$ for any $u\in\CD_V$ thus $\CD_\V\subset\CD(\CH)$ and it remains to prove that $\CD_V$ is dense in $L^2$ and that $(\CH,\CD_\V)$ is a closed operator. The map
\begin{equation}
u\mapsto u-\P_uX
\end{equation}
is continuous from $H^\sigma$ to itself for any $\sigma\le1+\kappa$ and it is invertible from $L^2$ to itself as a perturbation of the identity for $\|X\|_{C^{1+\kappa}}$ small enough. Using the truncated paraproduct $\P^N$ introduced in Appendix \ref{SectionPaleyLittlewood} which corresponds to a truncation of the low frequencies, there exists $N(\V)\ge1$ such that $\|P_u^NX\|_{H^{1+\kappa}}\le\frac{1}{2}\|u\|_{L^2}$ for $N\ge N(\V)$ hence
\begin{equation}
\Phi_N(u):=u-\P_u^NX
\end{equation}
is an invertible perturbation of the identity from $H^{1+\kappa}$ to itself, denote $\Gamma_N$ its inverse. The map $\Phi_N$ is a compact perturbation of the identity and since $\P-\P^N$ is a regularizing operator, we have
\begin{equation}
\CD_\V=\Phi^{-1}(H^2)=\Phi_N^{-1}(H^2)=\Gamma_N H^2
\end{equation}
hence $\CD_\V$ is parametrized by $H^2$ via the map $\Gamma_N$. The following proposition gives the needed continuity results on $\Gamma_N$. We denote as $W^{\sigma,p}$ the $L^p$ based Sobolev spaces associated to the Laplacian
\begin{equation}
W^{\sigma,p}=\{u\in D'(\IT)\ ;\ (1-\partial_x^2)^{\frac{\sigma}{2}}u\in L^p\}.
\end{equation}

\begin{proposition}
For $N\ge N(V)$, the application $\Gamma_N$ is continuous from $H^\sigma$ to itself and $C^\sigma$ to itself for any $\sigma\le1+\kappa$. It is also continuous from $L^\infty$ to itself and from $W^{\sigma,p}$ to itself for any $p\in[1,\infty)$ and $\sigma<1+\kappa$.
\end{proposition}

\begin{proof}
The map $\Phi_N$ is a perturbation of the identity in $H^\sigma$ with
\begin{equation}
\|(\text{Id}-\Phi_N)u\|_{H^\sigma}=\|\P_u^NX\|_{H^\sigma}\le\frac{1}{2}\|u\|_{L^2}
\end{equation}
for $N\ge N(V)$ and $\sigma\le1+\kappa$. Thus $\Phi_N:H^\sigma\to H^\sigma$ is invertible with inverse $\Gamma_N$ continuous. The results in Hölder spaces and $L^\infty$ follows from the same type of computations. For $W^{\sigma,p}$ with $p\in[1,\infty)$ and $\sigma<1+\kappa$, one needs to use continuity of the paraproduct in general Besov spaces with the inclusion $B_{p,2}^\sigma\hookrightarrow W^{\sigma,p}\hookrightarrow B_{p,\infty}^\sigma$.
\end{proof}

In particular, the continuity of $\Gamma_N$ gives that the domain is dense in $H^{1+\kappa}$ thus in $L^2$.

\begin{corollary}
The space $\CD_\V$ is dense in $H^{1+\kappa}$.
\end{corollary}

\begin{proof}
Let $f\in H^{1+\kappa}$. Then 
\begin{equation}
f=(\Gamma_N\circ\Phi_N)(f)=\Gamma_N g
\end{equation} 
with $g=\Phi_N(f)\in H^{1+\kappa}$. Since $H^2$ is dense in $H^{1+\kappa}$, there exists a familly $(g_\eps)_\eps\subset H^2$ such that
\begin{equation}
\lim_{\eps\to0}\|g-g_\eps\|_{H^{1+\kappa}}=0.
\end{equation}
Since $\Gamma_N$ is continuous from $H^{1+\kappa}$ to itself, we get
\begin{equation}
\lim_{\eps\to0}\|f-\Gamma_N g_\eps\|_{H^{1+\kappa}}=0.
\end{equation}
\end{proof}

We have the explicit formula
\begin{equation}
\CH\Gamma_N v=-\Delta v+\P_\V u+\PI(u,\V)+2\P_{\nabla u}\nabla X+\P_{\Delta u}X+\P_u\langle\V,1\rangle+(\P_u-\P_u^N)\xi
\end{equation}
for $v\in H^2$ and $u=\Gamma_Nv$, where the product rule gives
\begin{equation}\label{ExpressionH}
[\P_u,\Delta]X=2\P_{\nabla u}\nabla X+\P_{\Delta u}X.
\end{equation}
We first prove the following lemma which controls $\CH\Gamma_N$ as a perturbation of the Laplacian. In the following, we will denote $\Gamma=\Gamma_N$ and keep the dependence on $N$ implicit to lighten the notation.

\begin{lemma}\label{LemmaComparisonH2}
There exists a constant $c>0$ such that
\begin{equation}
\frac{1}{2}\|v\|_{\CH^2}\le\|\CH u\|_{L^2}+c\|u\|_{L^2}
\end{equation}
for any $u\in\CD_V$ and $v=\Phi(u)$.
\end{lemma}

\begin{proof}
Let $u\in\CD_V=\Gamma H^2$ and denote $v=\Phi(u)\in H^2$. The expression of $\CH u=\CH\Gamma v$ yields
\begin{align}
\|\Delta v+\CH u\|_{L^2}&=\|\P_\V u+\PI(u,\V)+2\P_{\nabla u}\nabla X+\P_{\Delta u}X+\P_u\langle\V,1\rangle\|_{L^2}\\
&\lesssim\|\V\|_{C^{-1+\kappa}}\|u\|_{H^1}+\|\nabla X\|_{C^{\kappa}}\|\nabla u\|_{H^{-\kappa}}+\|X\|_{C^{1+\kappa}}\|\Delta u\|_{H^{-1-\kappa}}\\
&\lesssim\|\V\|_{C^{-1+\kappa}}\|v\|_{H^1}\\
&\le \frac{1}{2}\|v\|_{H^2}+c\|v\|_{L^2}
\end{align}
for a constant $c>0$ large enough depending on $\V$ again by standard interpolation. Using that $\Phi$ is continuous from $L^2$ to itself, we get
\begin{equation}
\|\Delta v+\CH u\|_{L^2}\le \frac{1}{2}\|v\|_{H^2}+c\|\Phi\|_{L^2\to L^2}\|u\|_{L^2}
\end{equation}
which completes the proof using
\begin{equation}
\|\Delta v\|_{L^2}\le\|\Delta v+\CH u\|+\|\CH u\|_{L^2}
\end{equation}
and the continuity of $\Gamma$ from $L^2$ to itself.
\end{proof}

This lemma states that the norm induced by $\Gamma$ is equivalent to the norm domain. In particular, this allows to prove that $\CD_V$ corresponds indeed to the domain of $\CH$.

\begin{proposition}
The operator $(\CH,\CD_\V)$ is a closed operator in $L^2$ thus $\CD_\V=\CD(\CH)$.
\end{proposition}

\begin{proof}
Let $(u_n)_n\subset\CD_V$ and $u,v\in L^2$ such that
\begin{equation}
\lim_{n\to\infty}\|u_n-u\|_{L^2}+\|\CH u_n-v\|_{L^2}=0.
\end{equation}
Using Lemma \ref{LemmaComparisonH2}, $u_n^\sharp:=\Gamma^{-1}u_n$ is a Cauchy sequence in $H^2$ thus converges to a function $u^\sharp\in H^2$. Since $\Gamma^{-1}$ is continuous from $L^2$ to itself, we have $u^\sharp=\Gamma^{-1}u$. One gets $v=\CH u$ with
\begin{align}
\|\CH u-v\|_{L^2}&\lesssim\|\CH u-\CH u_n\|_{L^2}+\|\CH u_n-v\|_{L^2}\\
&\lesssim\|u^\sharp-u_n^\sharp\|_{H^2}+\|u-u_n\|_{L^2}+\|\CH u_n-v\|_{L^2}
\end{align}
and the proof is complete.
\end{proof}

\begin{remark}
Instead of the paracontrolled calculus, one could use the exponential transform as in \cite{HairerLabbe15,MouzardOuhabaz23} or the generalized Sturm-Liouville theory as in \cite{DumazLabbe} and references therein to study the operator $\CH$. These representation of the operator $\CH$ involve a first order term which is problematic to obtain Strichartz inequalities. For example, the exponential transform gives
\begin{equation}
e^{-X}\CH e^X=-\partial_x^2-2\partial_xX\cdot\partial_x-|\partial_xX|^2
\end{equation}
hence the first order term in the equation. One of the major advantage of paracontrolled calculus is that this first order term does not appear.
\end{remark}

The Sobolev spaces associated to $\CH$ can be defined as the closure
\begin{equation}
\CD^\sigma:=\overline{\text{Vect}(e_n\ ;\ n\ge1)}^{\|\cdot\|_{\CD^\sigma}}
\end{equation}
with the norm
\begin{equation}
\|u\|_{\CD^\sigma}^2:=\sum_{n\ge1}(1+\lambda_n)^{\sigma}|\langle u,e_n\rangle|^2
\end{equation}
for $\sigma\in\IR$. The previous results give the form domain $\CD^1=H^1$ as for the Laplacian while the domain $\CD^2=\CD_V$ depends on $V$. 

\begin{proposition}
For $|\sigma|\le1+\kappa$, we have $\CD^\sigma=H^\sigma$.
\end{proposition}

\begin{proof}
Since the operator is closed, $\CD^2$ corresponds to the domain hence $\CD^2=\CD_V=\Gamma H^2$. One also has $\CD^0=L^2=H^0=\Gamma H^0$ thus interpolation gives
\begin{equation}
\CD^\sigma\subset \Gamma H^\sigma
\end{equation}
for any $\sigma\in(0,2)$. Since $\Gamma^{-1}=\Phi$ is also continuous, we get
\begin{equation}
\CD^\sigma=\Gamma H^\sigma
\end{equation}
and the result follows for $0\le\sigma\le1+\kappa$ from the fact that $\Gamma$ is invertible from $H^\sigma$ to itself under this condition. For negative exponent, the result follows from duality with $\CD^{-\sigma}=(\CD^\sigma)'$.
\end{proof}

This proves useful since a number of properties are natural to prove in the Sovolev spaces associated to $\CH$. An example is this continuity result for the Schrödinger propagator which follows from the conservation of $\CD^\sigma$ norm. In particular, this can not be obtained a priori with a perturbative argument.

\begin{corollary}
For any $t\in\IR$ and $|\sigma|\le1+\kappa$, the Schrödinger propagator $e^{it\CH}$ is continuous from $H^\sigma$ to itself. In particular, the same holds for $e^{it\CH^\sharp}$.
\end{corollary}

In order to obtain Strichartz inequalities for $\CH^\sharp$, the comparison with the Laplacian given by the following lemma is needed. This is exaclty where paracontrolled calculus comes into play and improve the naive bound
\begin{equation}
\|\CH u+\Delta u\|_{H^{-1+\kappa}}\lesssim\|u\|_{H^{1-\kappa+\delta}}
\end{equation}
for any $\delta>0$.

\begin{proposition}\label{PropPerturbsharp}
For any $\delta\in(0,\kappa]$, we have
\begin{equation}
\|\CH^\sharp v+\Delta v\|_{H^\delta}\lesssim\|v\|_{H^{1-\kappa+\delta}}.
\end{equation}
\end{proposition}

\begin{proof}
This follows directly from formula \eqref{ExpressionH}, that is
\begin{equation}
\CH^\sharp v+\Delta v=\P_\V u+\PI(u,\V)+2\P_{\nabla u}\nabla X+\P_{\Delta u}X+\P_u\langle\V,1\rangle
\end{equation}
for $u=\Gamma v$. Because of the resonant product
\begin{equation}
\PI(u,\V)=\PI(\P_uX,\V)+\PI(v,\V),
\end{equation}
one needs $v$ to be of regularity higher than $1-\kappa$ for $\CH^\sharp v$ to make sense, just as $\CH$, and we have $\PI(u,\V)\in H^\delta$ for $v\in H^{1-\kappa+\delta}$ for $\delta\le\kappa$. In any case, the roughest term is given by the noise, for example $\PI(\P_uX,\V)\in H^\kappa$ hence the result.
\end{proof}

\section{Strichartz inequalities and low regularity initial data}\label{SectionStrichartz}


A solution theory with initial data $u_0\in L^2$ only yields a solution in $L^2$ since the Schrödinger flow does not improve regularity hence the product $\V u$ is a priori singular. However following the previous section, the Hamiltonian
\begin{equation}
\CH=-\partial_x^2+V
\end{equation}
allows to consider $u(t)=e^{-it\CH}u_0$ as a solution in $L^\infty(\IR,L^2)$ which satisfies \eqref{LS} in $\CD^{-2}$. In particular, the space $\CD^2=\CD(\CH)$ does not contain smooth functions thus $\CD^{-2}=(\CD^2)^*$ is not a space of distributions. To solve the nonlinear equation \eqref{NLS} for low regularity initial data, we consider the mild formulation
\begin{equation}
u(t)=e^{-it\CH}u_0-i\lambda\int_0^te^{-i(t-s)\CH}|u(s)|^{m-2}u(s)\drm s
\end{equation}
with $u(0)=u_0\in H^\sigma$. Since $H^\sigma$ is an algebra for $\sigma>\frac{1}{2}$, this gives local well-posedness in $H^\sigma$ and global solution for $u\in H^1$ with the conservation of energy in the defocusing case $\lambda>0$. In this section, we prove Strichartz inequalities for $\CH$ to obtain local well-posedness for low regularity initial data, that is below the algebra condition. In the cubic case without potential, Bourgain obtained local well-posedness in $L^2$ thus global solution using the conservation of mass. Even in the cubic case, we have a positive critial threshold $\sigma_c>0$ which prevents us from obtaining global solution below $H^1$. Our argument is perturbative and relies on Bourgain's result
\begin{equation}
\|e^{it\Delta}\Delta_Nu_0\|_{L^6([0,1]\times\IT)}\lesssim 2^{N\eps}\|\Delta_Nu_0\|_{L^2(\IT)}
\end{equation}
proved in \cite{Bourgain93}, with $\Delta_N$ the Paley-Littlewood projector defined in Appendix \ref{SectionPaleyLittlewood} supported on frequencies less than $2^N$. Because of the roughness of the potential $\V$, we consider the conjugated operator
\begin{equation}
\CH^\sharp=\Gamma^{-1}\CH\Gamma
\end{equation}
with $\Gamma$ defined by the implicit relation $\Gamma u^\sharp=\P_{\Gamma u^\sharp}X+u^\sharp$ introduced in the previous section. We insist that $\Gamma$ is invertible but not unitary hence $\CH^\sharp$ is not self-adjoint. With the variable $u=\Gamma u^\sharp$, one gets the new equation
\begin{equation}
u^\sharp(t)=e^{-it\CH^\sharp}u_0-i\lambda\int_0^te^{-i(t-s)\CH^\sharp}\Gamma^{-1}\big(|\Gamma u^\sharp(s)|^{m-2}\Gamma u^\sharp(s)\big)\drm s
\end{equation}
with $u^\sharp(0)=\Gamma^{-1}(u_0)$. While the non-linearity seems more complicated, the linear part is better behaved since $\CH^\sharp$ is a better behaved perturbation of the Laplacian as explained. Our proof follows the idea of splitting the time interval into small frequency dependent pieces which goes back to Bahouri and Chemin \cite{BahourChemin99} and Tataru \cite{Tataru00}.

\medskip

\begin{theorem}\label{THMStrichartz}
For any potential $\V\in\CC^{-1+\kappa}$ with $\kappa\in(0,1)$, we have
\begin{equation}
\|e^{it\CH^\sharp}u_0\|_{L^6([0,1]\times\IT)}\lesssim \|u_0\|_{H^{\frac{1-\kappa}{6}+\eps}}
\end{equation}
for any $\eps>0$.
\end{theorem}

\medskip

\begin{remark}
The loss of derivative here is similar in spirit to the one of Burq, Gérard and Tzvetkov \cite{BGT} with the Laplacian on an arbitrary manifold and Mouzard and Zachhuber \cite{MZ} with the Anderson Hamiltonian on compact surfaces. This is due to the fact that our knowledge of the spectral properties of the operator is not precise enough while Bourgain's method based on Fourier's series makes full use of the explicit spectrum of the Laplacian.
\end{remark}

\medskip

\begin{proof}
We work with the Paley-Littlewood projectors $\Delta_j$ on an annulus of size $2^j$ in frequencies, see Appendix \ref{SectionPaleyLittlewood}. Let $k,j\ge0$ be fixed. For any $N\ge1$, we have
\begin{equation}
\|\Delta_je^{it\CH^\sharp}\Delta_kv\|_{L^6([0,1]\times\IT)}^6=\sum_{n=0}^{N-1}\|\Delta_je^{it\CH^\sharp}\Delta_kv\|_{L^6([t_n,t_{n+1}]\times\IT)}^6
\end{equation}
with $t_n=\frac{n}{N}$ for $0\le n\le N$. For $t\in[t_n,t_{n+1}]$, we have
\begin{align*}
e^{it\CH^\sharp}\Delta_kv&=e^{i(t-t_n)\CH^\sharp}e^{-it_n\CH^\sharp}\Delta_kv\\
&=e^{-i(t-t_n)\Delta}e^{-it_n\CH^\sharp}\Delta_kv+\int_{t_n}^te^{-i(t-s)\Delta}\big(\CH^\sharp+\Delta\big)e^{-it_n\CH^\sharp}\Delta_kv\drm s
\end{align*}
using the mild formulation for $e^{i(t-t_n)\CH^\sharp}$ with respect to the Laplacian hence
\begin{align*}
\|\Delta_je^{it\CH^\sharp}\Delta_kv\|_{L^6([0,1]\times\IT)}^6\le\sum_{n=0}^{N-1}&\|\Delta_je^{-i(t-t_n)\Delta}e^{-it_n\CH^\sharp}\Delta_kv\|_{L^6([t_n,t_{n+1}]\times\IT)}^6\\
&+\|\int_{t_n}^t\Delta_je^{-i(t-s)\Delta}\big(\CH^\sharp+\Delta\big)e^{-it_n\CH^\sharp}\Delta_kv\drm s\|_{L^6([t_n,t_{n+1}]\times\IT)}^6.
\end{align*}
A crucial point is that the projector $\Delta_j$ commutes with $e^{it\Delta}$ while this is not true for $e^{it\CH^\sharp}$. For the first term, we have
\begin{align*}
\|\Delta_je^{-i(t-t_n)\Delta}e^{-it_n\CH^\sharp}\Delta_kv\|_{L^6([t_n,t_{n+1}]\times\IT)}&=\|e^{-i(t-t_n)\Delta}\Delta_je^{-it_n\CH^\sharp}\Delta_kv\|_{L^6([t_n,t_{n+1}]\times\IT)}\\
&\lesssim 2^{j\eps}\|\Delta_je^{-it_n\CH^\sharp}\Delta_kv\|_{L^2}\\
&\lesssim2^{-j\delta}\|e^{-it_n\CH^\sharp}\Delta_kv\|_{H^{\eps+\delta}}\\
&\lesssim2^{-j\delta}\|\Delta_kv\|_{H^{\eps+\delta}}\\
&\lesssim2^{-j\delta}2^{-k\delta'}\|\Delta_kv\|_{H^{\eps+\delta+\delta'}}
\end{align*}
for any $\delta,\delta',\eps>0$ using Bernstein lemma and that the Sobolev spaces $H^\sigma$ associated to $\CH$ and $\Delta$ are equivalent for $\sigma\in[0,1+\kappa]$. For the second term, we have
\begin{align*}
\|\int_{t_n}^t\Delta_je^{-i(t-s)\Delta}\big(\CH^\sharp+\Delta\big)e^{-it_n\CH^\sharp}\Delta_kv\drm s\|_{L^6([t_n,t_{n+1}]\times\IT)}&\le\int_{t_n}^{t_{n+1}}\|\Delta_j\big(\CH^\sharp+\Delta\big)e^{-it_n\CH^\sharp}\Delta_kv\|_{H^\eps}\drm s\\
&\lesssim N^{-1}2^{-j\sigma}\|\big(\CH^\sharp+\Delta\big)e^{-it_n\CH^\sharp}\Delta_kv\|_{H^{\eps+\sigma}}\\
&\lesssim N^{-1}2^{-j\sigma}\|e^{-it_n\CH^\sharp}\Delta_kv\|_{H^{\eps+\sigma+1-\kappa}}\\
&\lesssim N^{-1}2^{-j\sigma}\|\Delta_kv\|_{H^{\eps+\sigma+1-\kappa}}\\
&\lesssim N^{-1}2^{-j\sigma}2^{-k\sigma'}\|\Delta_kv\|_{H^{\eps+\sigma+\sigma'+1-\kappa}}
\end{align*}
for any $\sigma,\sigma',\eps>0$ with the same arguments in addition to Proposition \ref{PropPerturbsharp} to control $\CH^\sharp+\Delta$ in a positive Sobolev space. We get
\begin{equation}
\|\Delta_je^{it\CH^\sharp}\Delta_kv\|_{L^6([0,1]\times\IT)}\lesssim N^{\frac{1}{6}}2^{-j\delta}2^{-k\delta'}\|\Delta_kv\|_{H^{\eps+\delta+\delta'}}+N^{-\frac{5}{6}}2^{-j\sigma}2^{-k\sigma'}\|\Delta_kv\|_{H^{\eps+\sigma+\sigma'+1-\kappa}}
\end{equation}
and we choose different parameters do deal with the sums $k\le j$ and $k>j$. For the first sum, consider
\begin{equation}
\delta=\frac{1}{6}\gamma'+\kappa,\delta'>0,\sigma>0,\sigma'>0,N=2^{\gamma'j},\gamma=\frac{5}{6}\gamma',\gamma'=1-\kappa
\end{equation}
which gives
\begin{align*}
\|\sum_{k\le j}\Delta_je^{it\CH^\sharp}\Delta_kv\|_{L^6([0,1]\times\IT)}&\lesssim\sum_{j\ge0}\sum_{k\le j}N^{\frac{1}{6}}2^{-j\delta}2^{-k\delta'}\|\Delta_kv\|_{H^{\eps+\delta+\delta'}}+N^{-\frac{5}{6}}2^{-j\sigma}2^{-k\sigma'}\|\Delta_kv\|_{H^{\eps+\sigma+\sigma'+1-\kappa}}\\
&\lesssim\sum_{j\ge0}N^{\frac{1}{6}}2^{-j\delta}\|\Delta_{\le j}v\|_{H^{\eps+\delta+\delta'}}+N^{-\frac{5}{6}}2^{-j\sigma}\|\Delta_{\le j}v\|_{H^{\eps+\sigma+\sigma'+1-\kappa}}\\
&\lesssim\sum_{j\ge0}N^{\frac{1}{6}}2^{-j\delta}\|\Delta_{\le j}v\|_{H^{\eps+\delta+\delta'}}+N^{-\frac{5}{6}}2^{-j\sigma}2^{\gamma j}\|\Delta_{\le j}v\|_{H^{\eps+\sigma+\sigma'+1-\kappa-\gamma}}\\
&\lesssim\sum_{j\ge0}2^{-j\kappa}\|\Delta_{\le j}v\|_{H^{\frac{1}{6}\gamma'+\eps+\kappa+\delta'}}+2^{-\frac{5}{6}\gamma'j}2^{-j\sigma}2^{\frac{5}{6}\gamma'j}\|\Delta_{\le j}v\|_{H^{\eps+\sigma+\sigma'+1-\kappa-\frac{5}{6}\gamma'}}\\
&\lesssim\|v\|_{H^{\frac{1}{6}\gamma'+\eps+\kappa+\delta'}}+\|v\|_{H^{1-\kappa-\frac{5}{6}\gamma'+\eps+\sigma+\sigma'}}\\
&\lesssim\|v\|_{H^{\frac{1-\kappa}{6}+\eps'}}
\end{align*}
for any $\eps'>0$. For the sum $k>j$, we take
\begin{equation}
\delta>0,\delta'=\frac{1}{6}\gamma'+\kappa,\sigma>0,\sigma'>0,N=2^{\gamma'k},\gamma=\frac{5}{6}\gamma',\gamma'=1-\kappa
\end{equation}
hence
\begin{align*}
\|\sum_{k>j}\Delta_je^{it\CH^\sharp}\Delta_kv\|_{L^6([0,1]\times\IT)}&\lesssim\sum_{k\ge0}\sum_{j<k}N^{\frac{1}{6}}2^{-j\delta}2^{-k\delta'}\|\Delta_kv\|_{H^{\eps+\delta+\delta'}}+N^{-\frac{5}{6}}2^{-j\sigma}2^{-k\sigma'}\|\Delta_kv\|_{H^{\eps+\sigma+\sigma'+1-\kappa}}\\
&\lesssim\sum_{k\ge0}2^{\frac{1}{6}\gamma'k}2^{-k\delta'}\|\Delta_kv\|_{H^{\eps+\delta+\delta'}}+2^{-\frac{5}{6}\gamma'k}2^{-k\sigma'}2^{k\gamma}\|\Delta_kv\|_{H^{\eps+\sigma+\sigma'+1-\kappa-\gamma}}\\
&\lesssim\|v\|_{H^{\frac{1}{6}\gamma'+\eps'}}+\|v\|_{H^{1-\kappa-\frac{5}{6}\gamma'+\eps'}}\\
&\lesssim\|v\|_{H^{\frac{1-\kappa}{6}+\eps'}}
\end{align*}
for any $\eps'>0$ which completes the proof.
\end{proof}

In order to deal with general power in equation \eqref{NLS}, one can interpolate between this result with the $L^\infty([0,1],L^2)$ bound given by the conservation of mass, this is the content of the following corollary.

\medskip

\begin{corollary}\label{StrichartzInter}
Let $\V\in\CC^{-1+\kappa}$ with $\kappa\in(0,1)$. For any $\theta\in[0,1]$, we have
\begin{equation}
\|e^{it\CH^\sharp}u_0\|_{L^p([0,1],L^q)}\lesssim \|u_0\|_{L^2}^{1-\theta}\|u\|_{H^{\frac{1-\kappa}{6}+\eps}}^\theta
\end{equation}
for any $t\in\IR$ and $\eps>0$ with $p\in[6,\infty]$ and $q\in[2,6]$ such that
\begin{equation}
\frac{1}{p}=\frac{\theta}{6}\quad\text{and}\quad\frac{1}{q}=\frac{1-\theta}{2}+\frac{\theta}{6}.
\end{equation}
\end{corollary}



Using the dispersive properties of the equation allows to obtain local well-posedness with any $\kappa\in(0,1)$ for initial data below the algebra condition with critical threshold
\begin{equation}
\sigma_\kappa(m)=\frac{1}{3}-\frac{\kappa}{6}
\end{equation}
when $m\le 8$ and
\begin{equation}
\sigma_\kappa(m)=\frac{4-\kappa}{6}-\frac{2}{m-2}
\end{equation}
when $m>8$. 

\medskip

\begin{theorem}\label{THMlocalWP}
Let $\V\in\CC^{-1+\kappa}$ and $m\ge2$ an integer. For any $\sigma\in(\sigma_\kappa(m),1+\kappa]$, there exists a unique solution $u\in C([0,T],H^\sigma)$ to equation \eqref{NLS} for $u_0\in H^\sigma$ up to a time $T$ of order $\|u_0\|_{H^\sigma}^{-\frac{6(m-2)}{m-8}}$ and we have
\begin{equation}
\|u\|_{L^6([0,T],L^\infty)}+\sup_{t\in[0,T]}\|u(t)\|_{H^\sigma}\le 2\|u_0\|_{H^\sigma}.
\end{equation}
Moreover, the solution is locally Lipschitz with respect to the initial data from $\dst H^\sigma$ to $C([0,T],H^\sigma)\cap L^6([0,T],L^\infty)$.
\end{theorem}

\medskip

\begin{proof}
To solve equation \eqref{NLS}, we consider the new variable $v(t)=\Gamma^{-1}u(t)$ which satisfies the equation
\begin{equation}
i\partial_tv=\CH^\sharp v+\lambda\Gamma^{-1}(|\Gamma v|^{m-2}\Gamma v)
\end{equation}
with initial data $v(0)=\Gamma^{-1}u_0\in\Gamma^{-1}H^\sigma=H^\sigma$ since $\sigma\le1+\kappa$. We prove existence and uniqueness of solutions for this equation which gives a unique solution to the initial equation using that $\Gamma$ is continuous from $H^\sigma$ to itself. Note that this is direct for $\sigma\in(\frac{1}{2},1+\kappa]$ using the continuity of $\Gamma$ and its inverse and that $H^\sigma$ is an algebra. For the rest of the proof, we consider $\sigma\le\frac{1}{2}$. For $T>0$, let $\CS_T$ be the solution space 
\begin{equation}
\CS_T=C([0,T],H^\sigma)\cap L^6([0,T],W^{\frac{1}{6}+\eps,6})
\end{equation}
with the norm
\begin{equation}
\|v\|_{\CS_T}:=\sup_{t\in[0,T]}\|v(t)\|_{H^\sigma}+\|v\|_{L^6([0,T],W^{\frac{1}{6}+\eps,6})}
\end{equation} 
for $2\eps\in(0,\sigma-\sigma_\kappa]$. We prove that the map
\begin{equation}
\Phi(v)(t):=e^{-it\CH^\sharp}v_0-i\lambda\int_0^te^{-i(t-s)\CH^\sharp}\Gamma^{-1}\big(|\Gamma v(s)|^{m-2}\Gamma v(s)\big)\drm s
\end{equation}
is a contraction on a ball of $\CS_T$ for $T$ small enough. We have
\begin{align*}
\|\Phi(v)(t)\|_{H^\sigma}&\lesssim\|v_0\|_{H^\sigma}+\int_0^t\||\Gamma v(s)|^{m-2}\Gamma v(s)\|_{H^\sigma}\drm s\\
&\lesssim\|v_0\|_{H^\sigma}+\int_0^t\|v(s)\|_{H^\sigma}\|v(s)\|_{L^\infty}^{m-2}\drm s\\
&\lesssim\|v_0\|_{H^\sigma}+\|v\|_{L^\infty([0,t],H^\sigma)}T^{1-\frac{m-2}{6}}\|v\|_{L^6([0,t],L^\infty)}^{m-2}
\end{align*}
using that $\Gamma$ is continuous from $H^\sigma$ to itself, from $L^\infty$ to itself and Hölder inequality for $m<8$ with $t\in[0,T]$. Since $\sigma\le\frac{1}{2}$, the space $H^\sigma$ is not an algebra and one needs to control the $L^\infty$ norm, this is where Strichartz inequalities are crucial with the embedding
\begin{equation}
W^{\frac{1}{q}+\eps,q}\hookrightarrow L^\infty
\end{equation}
which gives the bound
\begin{equation}
\sup_{t\in[0,T]}\|\Phi(v)(t)\|_{H^\sigma}\lesssim \|v_0\|_{H^\sigma}+T^{\frac{8-m}{6}}\|v\|_{\CS_T}^{m-1}.
\end{equation} 
Using that the spaces $\CD^\delta$ and $H^\delta$ are equivalent for $|\delta|\le1+\kappa$ and the continuity results on $\Gamma$, Theorem \ref{THMStrichartz} gives
\begin{equation}
\|e^{-it\CH^\sharp}v\|_{L^6([0,T],W^{\delta,6})}\lesssim\|v\|_{H^{\frac{1-\kappa}{6}+\eps+\delta}}
\end{equation}
hence
\begin{align}
\|\Phi(v)\|_{L^6([0,t],W^{\frac{1}{6}+\eps,6})}&\lesssim\|v_0\|_{H^{\frac{1-\kappa}{6}+2\eps+\frac{1}{6}}}+\int_0^t\|v(s)\|_{H^{\frac{1-\kappa}{6}+2\eps+\frac{1}{6}}}\|v(s)\|_{L^\infty}^{m-2}\drm s\\
&\lesssim\|v_0\|_{H^{\frac{1-\kappa}{6}+2\eps+\frac{1}{6}}}+\|v\|_{L^\infty([0,t],H^{\frac{1-\kappa}{6}+2\eps+\frac{1}{6}})}T^{1-\frac{m-2}{6}}\|v\|_{L^6([0,t],W^{\frac{1}{6}+\eps,6})}^{m-2}\\
&\lesssim \|v_0\|_{H^\sigma}+T^{\frac{8-m}{6}}\|v\|_{\CS_T}^{m-1}
\end{align}
since $\sigma\ge\sigma_\kappa+2\eps$. This gives
\begin{equation}
\|\Phi(v)\|_{\CS_T}\lesssim\|v_0\|_{H^\sigma}+T^{\frac{8-m}{6}}\|v\|_{\CS_T}^{m-1}
\end{equation}
hence for any $A\ge0$ large enough depending on the initial condition, the ball centered in $0$ of size $A$ is stable by $\Phi$ for $T>0$ small enough. Similar computations give
\begin{equation}
\|\Phi(v)-\Phi(v')\|_{\CS_T}\lesssim T^{\frac{8-m}{6}}(1+\|v\|_{\CS_T}^{m-2}+\|v'\|_{\CS_T}^{m-2})\|v-v'\|_{\CS_T}
\end{equation}
for $v,v'\in\CS_T$ with $v(0)=v'(0)$. This gives that for $v_0\in\CH^\sigma$ and $A=A(v_0)$ large enough, the map $\Phi$ is a contraction on
\begin{equation*}
\big\{v\in\CS_T\ ;\ v(0)=v_0\text{ and }\|v\|_{\CS_T}\le A\big\}
\end{equation*}
for $T=T(A,v_0)$ small enough, that is $T<CA^{-\frac{6(m-2)}{(m-8)}}$ for a constant $C>0$. This also gives that the flow map $u_0\mapsto u$ is Lipschitz continuous from bounded subsets of $H^\sigma$ to $C([0,T],H^\sigma)\cap L^6([0,T],L^\infty)$.

\medskip

To deal with $m\ge8$, one needs to control higher integrability in time. Corollary \ref{StrichartzInter} gives
\begin{equation}
\|e^{it\CH^\sharp}u\|_{L^{p-2}([0,T],W^{\frac{1}{q}+\eps,q})}\lesssim\|u\|_{\CH^{\frac{1-\kappa}{6}+\frac{1}{q}+\eps}}
\end{equation}
for any $m>8$ and $\frac{1}{q}=\frac{1-\theta}{2}+\frac{\theta}{6}$ with $\theta=\frac{6}{m-2}$. Similar computations as before give
\begin{equation}
\|v(t)\|_{L^{m-2}([0,T],W^{\frac{1}{q}+\eps,q})}\lesssim\|u_0\|_{H^{\frac{1-\kappa}{6}+\frac{1}{q}+\eps}}+\|v\|_{L^\infty([0,T],H^{\frac{1-\kappa}{6}+\frac{1}{q}+\eps})}\|v\|_{L^{m-2}([0,T],L^\infty)}^{m-2}
\end{equation}
hence the condition 
\begin{equation}
\sigma>\frac{1-\kappa}{6}+\frac{1}{q}.
\end{equation}
With the relation between $q$ and $m$, we have
\begin{equation}
\sigma_\kappa(m)=\frac{1-\kappa}{6}+\frac{1}{2}-\frac{2}{m-2}
\end{equation}
when $m>8$.
\end{proof}

It is a priori not clear if this local well-posedness result is optimal in general. The important fact is that the mild formulation
\begin{equation}
u(t)=e^{-it\Delta}u_0+i\int_0^te^{-i(t-s)\Delta}\V u(s)\drm s-i\lambda\int_0^te^{-i(t-s)\Delta}|u(s)|^{m-2}u(s)\drm s
\end{equation}
can not be used to deal with low regularity initial data due to the singular product $\V u$. This prevents the use of methods based on explicit space-time Fourier transform as done by Bourgain in \cite{Bourgain93} to prove local well-posedness since one has to deal with the Hamiltonian $\CH$ and its Schrödinger group $e^{it\CH}$ which are not explicit. Nevertheless our proof of local well-posedness does not rely too much on the form of the non-linearity and one can consider the truncated equation
\begin{equation}\label{NLSN}
i\partial_tu_N=\CH u_N+\lambda\Pi_N(|u_N|^{m-2}u_N)
\end{equation}
with $u_N(0)=\Pi_Nu_0$. Here $\Pi_N$ denotes the spectral projector associated to $\CH$ on the space $E_N=\text{Vect}(e_1,\ldots,e_N)$ hence equation \eqref{NLSN} is a finite dimensional system. Since $\Gamma$ is an isomorphism, the space $F_N:=\Gamma E_N$ is also of finite dimension. We have the following proposition which states that the Galerkin approximation converges to the maximal solution on any time interval.

\medskip

\begin{proposition}\label{ProplocalWPapprox}
Let $\V\in\CC^{-1+\kappa}$ with $\kappa\in(0,1)$ and $m<8$ an integer. For any $\sigma\in(\sigma_\kappa,\frac{1}{2})$, let $u$ be the unique maximal solution $u\in C([0,T^*(u_0)),H^\sigma)$ to equation \eqref{NLS} for $u_0\in H^\sigma$. Then for any $T< T^*(u_0)$ and $\sigma'<\sigma$, there exists a constant $C>0$ independent of $N$ such that
\begin{equation}
\|u-u_N\|_{L^6([0,T],L^\infty)}+\sup_{t\le T}\|u(t)-u_N(t)\|_{H^{\sigma'}}\le CN^{\sigma'-\sigma}
\end{equation}
for $N$ large enough and $u_N$ the solution to the truncated equation \eqref{NLSN} with initial data $\Pi_Nu_0$.
\end{proposition}

\medskip

\begin{proof}
The same fixed point argument for the truncated equation \eqref{NLSN} gives an existence up to a time $T=CA^{-c}>0$ for $\|u_N(0)\|_{H^\sigma}\le A$ with uniform bound
\begin{equation}
\|u_N\|_{L^6([0,T],L^\infty)}+\sup_{t\in[0,T]}\|u_N(t)\|_{H^\sigma}\le 2A.
\end{equation}
Similar computations also gives
\begin{equation}
\|u-u_N\|_{L^6([0,T],L^\infty)}+\sup_{t\in[0,T]}\|u(t)-u_N(t)\|_{H^\sigma}\le C_A\|u(0)-u_N(0)\|_{H^\sigma}
\end{equation}
for a constant $C_A>0$ since the solution is locally Lipschitz continuous with respect to the initial data. 

\smallskip

Now consider the global solution $u$ to equation \eqref{NLS} defined on $[0,T^*(u_0))$ with $T^*(u_0)\in(0,\infty]$ for $u_0\in H^\sigma$ and $\sigma'<\sigma$. Consider also the solution $u_N$ to the truncated equation \eqref{NLSN} with initial data $u_N(0)=\Pi_Nu_0$. Let $T<T^*(u_0)$ and
\begin{equation}
B:=b\sup_{t\in[0,T]}\|u(t)\|_{H^\sigma}+1
\end{equation}
with $b=\max(1,\|\Pi_N\|_{H^\sigma\to H^\sigma})>0$. Then we have $\|u_0\|_{H^\sigma}\le B$ and $\|\Pi_Nu_0\|_{H^\sigma}\le B$ thus
\begin{equation}
\|u-u_N\|_{L^6([0,\tau],L^\infty)}+\sup_{t\in[0,\tau]}\|u(t)-u_N(t)\|_{H^\sigma}\le C_B \|u_0-\Pi_Nu_0\|_{H^\sigma}
\end{equation}
for $\tau=CB^{-c}>0$ and $C_B>0$. Since $\sigma'<\sigma$, we get
\begin{equation}
\sup_{t\in[0,\tau]}\|u(t)-u_N(t)\|_{H^{\sigma'}}\le C_B N^{\sigma'-\sigma}\|u_0\|_{H^\sigma}
\end{equation}
thus
\begin{equation}
\|u_N(\tau)\|_{H^{\sigma'}}\le\|u(\tau)\|_{H^{\sigma'}}+C_B N^{\sigma'-\sigma}\|u_0\|_{H^\sigma}.
\end{equation}
For $N$ large enough, we get
\begin{equation}
\|u_N(\tau)\|_{H^{\sigma'}}\le\|u(\tau)\|_{H^{\sigma}}+1\le B.
\end{equation}
Thus we have $\|u(\tau)\|_{H^{\sigma'}}\le B$ as well as $\|u_N(\tau)\|_{H^{\sigma'}}\le B$ and the same argument gives 
\begin{equation}
\|u-u_N\|_{L^6([\tau,2\tau],L^\infty)}+\sup_{t\in[\tau,2\tau]}\|u(t)-u_N(t)\|_{H^{\sigma'}}\le C_B\|u(\tau)-u_N(\tau)\|_{H^{\sigma'}}
\end{equation}
thus
\begin{equation}
\sup_{t\in[0,2\tau]}\|u(t)-u_N(t)\|_{H^{\sigma'}}\le (1+C_B)C_BN^{\sigma'-\sigma}\|u_0\|_{H^\sigma}.
\end{equation}
Again, taking $N$ large enough yields
\begin{equation}
\|u_N(2\tau)\|_{H^{\sigma'}}\le\|u(2\tau)\|_{H^{\sigma'}}+1\le B
\end{equation}
thus a finite number of repition of this argument yields
\begin{equation}
\|u-u_N\|_{L^6([0,T],L^\infty)}+\sup_{t\in[0,T]}\|u(t)-u_N(t)\|_{H^{\sigma'}}\le CN^{\sigma'-\sigma}\|u_0\|_{H^\sigma}
\end{equation}
for a constant $C>0$.

\end{proof}

\section{Invariance of the Gibbs measure and global well-posedness}\label{SectionMeasure}



A natural question is the existence of an invariant measure for equation \eqref{NLS}, it goes back to Lebowitz, Rose and Speer \cite{LRS} and Bourgain \cite{Bourgain94}. The equation being Hamiltonian, it is expected that the formal Gibbs measure
\begin{equation}
\nu(\drm u)=\frac{1}{Z_\nu}e^{-\CE(u)}\prod_{x\in\IT}\drm u(x)
\end{equation}
with energy
\begin{equation}
\CE(u)=\int_\IT|\partial_xu(x)|^2\drm x+\int_{\IT}|u(x)|^2\V(\drm x)+\lambda\int_\IT|u(x)|^m\drm x
\end{equation}
for $u\in H^1$ leaves the dynamic invariant. This is only formal since the Lebesgue measure in infinite dimension does not exist and considering only the quadratic part, one gets the Gaussian measure
\begin{equation}
\mu(\drm u)=\frac{1}{Z_\mu}\exp\Big(-\int_\IT|\partial_xu(x)|^2\drm x-\int_{\IT}|u(x)|^2\V(\drm x)\Big)\prod_{x\in\IT}\drm u(x)
\end{equation}
which can be understood as the Gaussian measure with covariance function given by the kernel of $(-\partial_x^2+\V)^{-1}$. However, this is a positive operator only if the smallest eigenvalues $\lambda_1$ is positive hence in general one has to shift the operator to obtain a probability measure with
\begin{equation}
\mu(\drm u)=\frac{1}{Z_\mu}e^{-\big\langle (\CH-\lambda_1+1)u,u\big\rangle}\prod_{x\in\IT}\drm u(x)
\end{equation}
which is well-defined, see for example Da Prato's book \cite{DaPrato} for an introduction to Gaussian measure in infinite dimensions. In order to keep the same measure $\nu$, we get
\begin{equation}
\nu(\drm u)=\frac{1}{Z}e^{(1-\lambda_1)\int_\IT|u(x)|^2\drm x-\lambda\int_\IT|u(x)|^m\drm x}\mu(\drm u)
\end{equation}
with the normalisation constant $Z=\|e^{(1-\lambda_1)\int_\IT|u(x)|^2\drm x-\lambda\int_\IT|u(x)|^m\drm x}\|_{L^1(\mu)}$. The measure $\mu$ corresponds to the law of the Gaussian Free Field (GFF) associated to $\CH-\lambda_1+1$, that is the random serie
\begin{equation}
\phi(x)=\sum_{n\ge1}\frac{\gamma_n}{\sqrt{\lambda_n-\lambda_1+1}}e_n(x)
\end{equation}
with $(\gamma_n)_{n\ge1}$ a sequence of independent and identically distributed standard Gaussian random variables $\mathcal{N}(0,1)$. The unshifted case $\V=0$ corresponds to the law of the Brownian bridge, $\mu$ describes here a model of random bridge in the irregular environment $\V$. The fact that the measure is supported in $H^{\frac{1}{2}-\eps}(\IT)$ follows immediatly from the spectral representation of $\CH$, we prove Hölder regularity here.

\medskip

\begin{proposition}\label{PropSupportmu}
For any $\eps>0$, the measure $\mu$ is supported in $C^{\frac{1}{2}-\eps}$.
\end{proposition}

\medskip

\begin{proof}
Let $\phi$ be a random field of law $\mu$, that is a centered Gaussian field with covariance given by
\begin{equation}
\IE\big[\phi(x)\phi(y)\big]=G(x,y)
\end{equation}
for $x,y\in\IT$ and $G$ the Green function associated to $\CH-\lambda_1+1$. We have
\begin{equation}
\IE\big[|\phi(x)-\phi(y)|^2\big]=G(x,x)+G(y,y)-2G(x,y)
\end{equation}
since $G$ is symmetric hence the statement comes down to a regularity estimates on $G$. We have the semigroup representation
\begin{equation}
(\CH-\lambda_1+1)^{-1}=\int_0^1 e^{-t\CH}e^{(\lambda_1-1)t}\drm t+\CH^{-1}e^{-\CH}e^{\lambda_1-1}
\end{equation}
which gives
\begin{equation}
G(x,y)=\int_0^1p_t(x,y)e^{(\lambda_1-1)t}\drm t+r_1(x,y)e^{\lambda_1-1}
\end{equation}
with $p_t(x,y)$ the heat kernel associated to $e^{-t\CH}$ and $r_1(x,y)$ the kernel of $\CH^{-1}e^{-\CH}$. One has $\delta_x\in H^{-\frac{1}{2}-\eps}$ for any $\eps>0$ with a continuous dependance with respect to $x\in\IT$. This implies 
\begin{equation}
p_t(x,y)=(e^{-t\CH}\delta_x)(y)\quad\text{and}\quad r_1(x,y)=(\CH^{-1}e^{-\CH}\delta_x)(y)
\end{equation}
belongs to $\CD^\infty\subset C^{1+\kappa}$ in each variable uniformly with respect to the other. For $t>0$ small, the Lipschitz norm of $p_t(x,\cdot)$ diverges as $t$ goes to $0$. Schauder estimates in $\CD^\sigma$ follows directly from the boundeness from below of the spectrum of $\CH$, indeed
\begin{align}
\|p_t(x,\cdot)\|_{\CD^\sigma}^2&=\sum_{n\ge1}(1+\lambda_n)^\sigma e^{-t\lambda_n}|e_n(x)|^2\\
&=\sum_{n\ge1}(1+\lambda_n)^{\sigma+\frac{1}{2}+\eps}e^{-t\lambda_n}(1+\lambda_n)^{-\frac{1}{2}-\eps}|e_n(x)|^2\\
&\lesssim t^{-\sigma-\frac{1}{2}-\eps}\sum_{n\ge1}(1+\lambda_n)^{-\frac{1}{2}-\eps}|e_n(x)|^2\\
&\lesssim t^{-\sigma-\frac{1}{2}-\eps}\|\delta_x\|_{\CD^{-\frac{1}{2}-\eps}}^2
\end{align}
for any $\eps>0$ and $\sigma>-\frac{1}{2}-\eps$. Since $\CD^{-\frac{1}{2}-\eps}=H^{-\frac{1}{2}-\eps}$ for $\eps>0$ small enough, this gives
\begin{equation}
\big\|\int_0^1p_t(x,\cdot)e^{(\lambda_1-1)t}\drm t\big\|_{\CD^\sigma}\lesssim\int_0^1t^{-\frac{\sigma}{2}-\frac{1}{4}-\eps}\|\delta_x\|_{H^{-\frac{1}{2}-2\eps}}\drm t<\infty
\end{equation}
for $\sigma<\frac{3}{2}-2\eps$. Using Besov injection in one dimension, we get
that $G$ is of Hölder regularity $1-\delta$ for any $\delta>0$ in each coordinates hence
\begin{equation}
\IE\big[|\phi(x)-\phi(y)|^2\big]\lesssim|x-y|^{1-\delta}.
\end{equation}
Since $\phi$ is a Gaussian random field, the Hölder regularity follows from Kolmogorov theorem, see for example Theorem $3.3.16$ from Strook's book \cite{Strook1993}.
\end{proof}

Since the measure is supported on continuous functions, the potential energy $\|u\|_{L^m(\IT)}^m$ is finite for $\mu$-almost all functions and the definition of the measure
\begin{equation}
\nu(\drm u)=\frac{1}{Z}\exp\Big((1-\lambda_1)\int_\IT|u(x)|^2\drm x-\lambda\int_\IT|u(x)|^m\drm x\Big)\mu(\drm u)
\end{equation}
amounts to proving $\exp\Big((1-\lambda_1)\int_\IT|u(x)|^2\drm x-\lambda\int_\IT|u(x)|^m\drm x\Big)\in L^1(\mu)$. In the defocusing case $\lambda>0$, this follows from the fact that $\mu$ is a probability measure since the exponential is bounded for $m>2$ using the injection $L^m\hookrightarrow L^2$. Indeed, one has
\begin{align}
\exp\Big((1-\lambda_1)\int_\IT|u(x)|^2\drm x-\lambda\int_\IT|u(x)|^m\drm x\Big)&\le\exp\Big((1-\lambda_1)\int_\IT|u(x)|^2\drm x-\lambda c_m\big(\int_\IT|u(x)|^2\drm x\big)^{\frac{m}{2}}\Big)\\
&\le \sup_{r>0}e^{(1-\lambda_1)r^2-\lambda c_mr^m}\\
&<\infty
\end{align}
for any $\lambda>0$ with $c_p>0$ a positive constant and $m>2$. In the focusing case, one needs to introduce a cut-off $B>0$ following \cite{Bourgain94,LRS} with the measure
\begin{equation}
\nu_B(\drm u)=\frac{\IDC_{\|u\|_{L^2}\le B}}{Z_B}\exp\Big((1-\lambda_1)\int_\IT|u(x)|^2\drm x-\lambda\int_\IT|u(x)|^m\drm x\Big)\mu(\drm u)
\end{equation}
since a scaling argument gives that the measure can not be finite without truncation even in the case $\V=0$. This appears as the most natural truncation since the mass is the only conserved quantity for \eqref{NLS} with low regularity initial data. The following proposition guarantees that $Z_B<\infty$.

\begin{proposition}\label{PropFiniteMeasure}
For $m<6$ and any $B>0$, we have 
\begin{equation}
\IDC_{\|u\|_{L^2}\le B}\exp\Big((1-\lambda_1)\int_\IT|u(x)|^2\drm x-\lambda\int_\IT|\varphi(x)|^m\drm x\Big)\in L^1(\mu). 
\end{equation}
The result still holds for $m=6$ and $B$ small enough.
\end{proposition}

\begin{proof}
Since
\begin{equation}
\IDC_{\|u\|_{L^2}\le B}e^{(1-\lambda_1)\int_\IT|u(x)|^2\drm x}\le e^{(1-\lambda_1)B^2},
\end{equation}
we only have to deal with the potential energy term. Our goal is to use Fernique's theorem which ensures
\begin{equation}
\IE\big[e^{\beta\|u\|_{H^\sigma}^2}\big]<\infty
\end{equation}
for any $\sigma<\frac{1}{2}$ and $\beta=\beta(\sigma)$ small enough, see for example Da Prato and Zabczyk's book \cite{DaPratoZabczyk2014}. Sobolev embedding gives $H^\sigma(\IT)\hookrightarrow L^p(\IT)$ for any $\sigma\ge\frac{p-2}{2p}$, that is
\begin{equation}
\int_\IT|u(x)|^m\drm x\le C\|u\|_{H^{\frac{m-2}{2m}}}^m
\end{equation} 
for a constant $C>0$. In order to use the previous exponential moments, we interpolate with
\begin{equation}
\int_\IT|u(x)|^m\drm x\le C\|u\|_{L^2}^{(1-\theta)m}\|u\|_{H^\sigma}^{\theta m}
\end{equation}
with $\theta\in(0,1)$ and $\sigma\theta=\frac{m-2}{2m}$. Since we need $\sigma<\frac{1}{2}$ as well as $\sigma m<2$ in order to conclude with Young inequality, this imposes $m<4$. To extend the result to $m<6$, one needs interpolation with spaces based on $L^q$ with $q>2$. Besov injections and interpolation give
\begin{align*}
\int_\IT|u(x)|^m\drm x&\lesssim\|u\|_{L^2}^2\|u\|_{L^\infty}^{m-2}\\
&\lesssim\|u\|_{L^2}^2\|u\|_{B_{q,q}^{\frac{1}{q}+\eps}}^{m-2}\\
&\lesssim\|u\|_{L^2}^2\|u\|_{B_{2,2}^0}^{(1-\theta)(m-2)}\|u\|_{B_{\infty,\infty}^\sigma}^{\theta(m-2)}
\end{align*}
for any $q\ge2,\eps>0,\theta\in(0,1)$ and 
\begin{align*}
\sigma\theta&=\frac{1}{q}+\eps,\\
\frac{1}{q}&=\frac{1-\theta}{2}.
\end{align*}
We get
\begin{equation}
\sigma\theta=\frac{1-\theta}{2}+\eps\quad\iff\quad\theta=\frac{1+2\eps}{1+2\sigma}
\end{equation}
and the condition $\sigma<\frac{1}{2}$ implies
\begin{equation}
\theta(m-2)>\frac{1+2\eps}{2}(m-2).
\end{equation}
Since our goal is to have $\theta(m-2)<2$, this yields the condition
\begin{equation}
\frac{1}{2}(m-2)<2\iff m<6
\end{equation}
which is indeed the optimal condition. We get
\begin{align*}
\IE\Big[\IDC_{\|u\|_{L^2}\le B}e^{\int_\IT|u(x)|^m\drm x}\Big]&\le\IE\Big[e^{CB^{2+(1-\theta)(m-2)}\|u\|_{C^\sigma}^{\theta(m-2)}}\Big]\\
&\le\IE\Big[e^{C\frac{1}{q'}B^{2q'+(1-\theta)(m-2)q'}+\frac{1}{q}\|u\|_{C^\sigma}^{\theta(m-2)q}}\Big]\\
&\le e^{C\frac{1}{q'}B^{2q'+(1-\theta)(m-2)q'}}\IE\Big[e^{\frac{C_\eps}{q}\|u\|_{L^2}^2+\frac{\eps}{q}\|u\|_{C^{\sigma'}}^2}\Big]\\
&\le e^{C\frac{1}{q'}B^{2q'+(1-\theta)(m-2)q'}+\frac{C_\eps}{q}B^2}\IE\Big[e^{\frac{\eps}{q}\|u\|_{C^{\sigma'}}^2}\Big]
\end{align*}
for $q=\frac{1}{\theta(m-2)}>2$, $q'$ its conjugated exponent, any $\eps>0$, a constant $C_\eps>0$ large enough and $\sigma<\sigma'<\frac{1}{2}$ together with Young inequality. Using Fernique's theorem and that $\mu$ is a Gaussian measure on the Banach space $C^{\sigma'}$, we get
\begin{equation}
\IE\Big[\IDC_{\|u\|_{L^2}\le B}e^{\int_\IT|u(x)|^m\drm x}\Big]<\infty
\end{equation}
for any $B\ge0$ and $m<6$. The results for $m=6$ follows for $B$ small enough.
\end{proof}

\begin{remark}\label{RemarkCriticalCase}
Since $(\IDC_{Y\le B}e^X)^q=\IDC_{Y\le B}e^{qX}$, this also implies the result for $L^q$ insteand of $L^1$ with a condition on $B$ depending on $q$ in the case $m=6$. An alternative definition of the measure could have been with a density depending on $\V$ with respect to the case $\V=0$ studied in \cite{Bourgain94,LRS,OhSosoeTolomeo22}. Denoting as $\nu_B^0$ the case $\V=0$, we have
\begin{equation}
\nu_B(\drm u)=\frac{1}{Z_\V}e^{-\langle \V u,u\rangle}\nu_B^0(\drm u)
\end{equation}
where $\langle\V u,u\rangle$ is almost surely well-defined. Indeed, the measure $\nu_B^0$ is supported in $C^{\frac{1}{2}-\eps}$ for any $\eps>0$ and we have
\begin{align}
|\langle\V u,u\rangle|&\lesssim\|\V\|_{C^{-1+\kappa}}\|u^2\|_{B_{1,1}^{1-\kappa}}\\
&\lesssim\|\V\|_{C^{-1+\kappa}}\|u\|_{B_{2,2}^{1-\frac{\kappa}{2}}}^2\\
&\lesssim\|\V\|_{C^{-1+\kappa}}\|u\|_{C^{\frac{1}{2}-\frac{\kappa}{2}}}^2
\end{align}
using duality, product and injection of Besov spaces from Proposition \ref{PropBesov}. One can also prove that $Z_\V$ is finite using the cut-off, this is also an alternative proof of Proposition \ref{PropSupportmu}. Since the optimal parameter $B$ is known in the case $\V=0$ from \cite{OhSosoeTolomeo22}, the same is true here. This construction of the measure is less suited to the study of its truncation associated to the operator $\CH$ and the Schrödinger equation.
\end{remark}

Since the measure $\mu$ is supported in $C^{\frac{1}{2}-\eps}$ for any $\eps>0$, this also holds for $\nu_B$ as it is absolutely continuous with respect to $\mu$ with the previous proposition. Then Theorem \ref{THMlocalWP} guarantees that there exists almost surely a local solution to \eqref{NLS} for initial data distributed as $\nu_B$ in the focusing case, or $\nu$ in the defocusing case. We now prove the the measure is invariant, thus allowing to construct global solutions as done by Bourgain \cite{Bourgain94}. However, we rely here on a different argument from Da Prato and Debussche in the parabolic case. Recall that the truncated dynamic
\begin{equation}
i\partial_tu_N=\CH u_N+\lambda\Pi_N(|u_N|^{m-2}u_N)
\end{equation}
with $u_N(0)=\Pi_Nu_0$ is an approximation as $N$ goes to infinity of equation \eqref{NLS}. Since this is a finite dimensional system, Liouville's Theorem states that the Lebesgue measure is invariant and the Hamiltonian being invariant, the projected measure $\nu_N:=\Pi_N\nu$ is an invariant measure for equation \eqref{NLSN} defined on the finite dimensional space
\begin{equation}
\Omega_N:=\Pi_NL^2.
\end{equation}
Using Proposition \ref{ProplocalWPapprox}, we prove that one can pass to the limit and get global solution with invariant measure for \eqref{NLS}.

\begin{theorem}
The measure $\nu$ is invariant under the flow of equation \eqref{NLSN} and the equation is globally well-posed for $\nu$-almost all initial data.
\end{theorem}

\begin{proof}
Let $u_0$ be a random initial data with law $\nu_B$ in the focusing case and $\nu$ in the defocusing case. Then there exists a maximal solution up to a random time $T^*(u_0)\in(0,+\infty]$ with the explosion criterion
\begin{equation}
T^*(u_0)<\infty\quad\implies\quad\sup_{t<T^*(u_0)}\|u(t)\|_{H^\sigma}=+\infty.
\end{equation}
Let $T>0$. For any $t<T\wedge T^*(u_0)$, the mild formulation for equation \eqref{NLS} gives
\begin{equation}
\sup_{t<T\wedge T^*(u_0)}\|u(t)\|_{H^\sigma}\lesssim\|u_0\|_{H^\sigma}+\int_0^{T\wedge T^*(u_0)}\|u(t)\|_{L^\infty}^{m-2}\|u(t)\|_{H^\sigma}\drm t
\end{equation}
hence Fubini-Tonelli Theorem gives
\begin{equation}
\IE\Big[\sup_{t\le T\wedge T^*(u_0)}\|u(t)\|_{H^\sigma}\Big]\lesssim1+\int_0^T\IE\Big[\|u(t)\|_{L^\infty}^{m-2}\|u(t)\|_{H^\sigma}\IDC_{t<T^*(u_0)}\Big]\drm t.
\end{equation}
Proposition \ref{ProplocalWPapprox} gives
\begin{equation}
\|u-u_N\|_{L^6([0,\tau],L^\infty)}+\sup_{t\le \tau}\|u(t)-u_N(t)\|_{H^\sigma}\le  CN^{\sigma-\sigma'}
\end{equation}
for $\tau<T^*(u_0)$ and $\sigma'<\sigma$ with $C=C(u_0)>0$ a random constant hence Fatou's lemma gives
\begin{equation}
\IE\Big[\|u(t)\|_{L^\infty}^{m-2}\|u(t)\|_{H^\sigma}\IDC_{t<T^*(u_0)}\Big]\le\liminf_N\IE\Big[\|u_N(t)\|_{L^\infty}^{m-2}\|u_N(t)\|_{H^\sigma}\IDC_{t<T^*(u_0)}\Big]
\end{equation}
for almost all $t\in\IR$. However, $u_N(t)$ has a law $\Pi_N\nu$ or $\Pi_N\nu_B$ independant of $t$ with needed norm with uniformly bounded moment with respect to $N$ hence
\begin{equation}
\sup_{N\ge1}\IE\Big[\|u_N(t)\|_{L^\infty}^{m-2}\|u_N(t)\|_{H^\sigma}\Big]<\infty.
\end{equation}
We get
\begin{equation}
\IE\Big[\sup_{t\le T\wedge T^*(u_0)}\|u(t)\|_{H^\sigma}\Big]\lesssim1+T
\end{equation}
which guarantees with the blow up criterion that $T^*(u_0)>T$ almost surely. Since $T>0$ is chosen arbitrary, we get that $T^*(u_0)=+\infty$ and the equation is globally well-posed for $\nu$-almost all initial data. The invariance of the measure follows from the invariance of the truncated measure with the almost sure convergence of $u_N$ to $u$ in $C([0,T],H^\sigma)$ for any $T>0$.

\end{proof}

This result is enough to have a well-defined flow defined on a set of full measure. Indeed, it gives a set
\begin{equation}
S_N=:\{\omega\in\Omega\ ;\ \sup_{-N\le t\le N}\|u(\omega,t)\|_{H^\sigma}<\infty\}
\end{equation}
of full measure on which solutions are defined up to time $|t|\le N$, the arguments also work for negative time. Then
\begin{equation}
S=\bigcap_{N\ge1}S_N
\end{equation}
yields a set of full measure on which a global flow is well-defined.

\appendix

\section{Paley-Littlewood theory}\label{SectionPaleyLittlewood}


We give here results on Paley-Littlewood theory needed in our work. We omit the proofs and refer for example to the book of Bahouri, Chemin and Danchin \cite{BCD}. Consider two functions $\chi,\rho:\IR^d\to\IR$ such that $\chi$ is supported in a ball, $\rho$ in an annulus
\begin{equation}
\chi(z)+\sum_{j\ge0}\rho(2^{-j}z)=1
\end{equation}
for all $z\in\IR^d$, $\text{supp}(\chi)\cap\text{supp}(\rho(2^{-j}\cdot))=\emptyset$ and $\text{supp}(\rho(2^{-i}\cdot))\cap\text{supp}(\rho(2^{-j}\cdot))=\emptyset$ for $|i-j|>1$, this is called a dyadic parition of the unity. The Paley-Littlewood blocks are defiend as
\begin{equation}
\Delta_{-1}u=\SF^{-1}\chi\SF u\quad\text{and}\quad\Delta_ju=\SF^{-1}\rho(2^{-j}\cdot)\SF u,j\ge0
\end{equation}
and we have
\begin{equation}
u=\sum_{j\ge-1}\Delta_ju.
\end{equation}
Following the idea that decay of Fourier coefficient measures spatial regularity, one can consider the Besov spaces defined by the norm
\begin{equation}
\|u\|_{B_{p,q}^\alpha}=\Big(\sum_{j\ge-1}2^{\alpha jq}\|\Delta_ju\|_{L^p}^q\Big)^{\frac{1}{q}}
\end{equation}
for $p,q\in[1,\infty]$ and $\alpha\in\IR$ with an ad hoc definition for $q=\infty$. One recovers the Sobolev spaces $H^\alpha=B_{2,2}^\alpha$ for $\alpha\in\IR$ and the Hölder spaces $C^\alpha=B_{\infty,\infty}^\alpha$ for $\alpha\in\IR^+\backslash\IN$. We have the following embeddings, duality and interpolation results.

\medskip

\begin{proposition}\label{PropBesov}
The Besov spaces have the following properties.
\begin{enumerate}
	\item[i)] For $\alpha\in\IR$ and $p_1,p_2,q_1,q_2\in[1,\infty]$ such that $p_1\le p_2,q_1\le q_2$, we have
	\begin{equation}
	B_{p_1,q_1}^\alpha\subset B_{p_2,q_2}^{\alpha-d(\frac{1}{p_1}-\frac{1}{p_2})}.
	\end{equation}
	\item[ii)] Let $\alpha_1,\alpha_2\in\IR$ such that $\alpha_1+\alpha_2>0$ and $p_1,p_2\in[1,\infty]$. Then for any $\kappa>0$, we have
	\begin{equation}
	\|uv\|_{B_{p,p}^{\alpha-\kappa}}\lesssim\|u\|_{B_{p_1,p_1}^{\alpha_1}}\|v\|_{B_{p_2,p_2}^{\alpha_2}}
	\end{equation}
	with $\alpha=\alpha_1\wedge\alpha_2$ and $\frac{1}{p}=\frac{1}{p_1}+\frac{1}{p_2}$.
	\item[iii)] Let $\alpha\in\IR$ and $p,q\in[1,\infty]$. We have
	\begin{equation}
	|\langle u,v\rangle|\lesssim\|u\|_{B_{p,q}^\alpha}\|v\|_{B_{p',q'}^{-\alpha}}
	\end{equation}
	with $1=\frac{1}{p}+\frac{1}{p'}=\frac{1}{q}+\frac{1}{q'}$.
	\item[iv)] For any $\theta\in(0,1)$, let $p,p_1,p_2,q,q_1,q_2\in[1,\infty]$ such that 
	\begin{equation}
	\frac{1}{p}=\frac{1-\theta}{p_1}+\frac{\theta}{p_2},\quad\frac{1}{q}=\frac{1-\theta}{q_1}+\frac{\theta}{q_2}
	\end{equation}
	and $\alpha,\alpha_1,\alpha_2\in\IR$ such that
	\begin{equation}
	\alpha=(1-\theta)\alpha_1+\theta\alpha_2.
	\end{equation}
	Then we have
	\begin{equation}
	\|u\|_{B_{p,q}^\alpha}\lesssim\|u\|_{B_{p_1,q_1}^{\alpha_1}}^{1-\theta}\|u\|_{B_{p_2,q_2}^{\alpha_2}}^\theta.
	\end{equation}
\end{enumerate}
\end{proposition}

\medskip

The Paley-Littlewood decomposition can be used to define the paraproduct
\begin{equation}
P_uv=\sum_{n<m-1}\Delta_nu\Delta_mv
\end{equation}
and resonant product
\begin{equation}
\PI(u,v)=\sum_{|n-m|\le1}\Delta_nu\Delta_mv,
\end{equation}
this goes back to Coifman and Meyer \cite{CoifmanMeyer} and Bony \cite{Bony}. One can describe a produt as
\begin{equation}
uv=\P_uv+\PI(u,v)+\P_vu
\end{equation}
and each operator satisfies the following continuity result.

\medskip

\begin{proposition}\label{PropParaproduct}
For $\alpha,\beta\in\IR$, we have
\begin{equation}
\|P_uv\|_{C^{\alpha\wedge0+\beta}}\lesssim\|u\|_{C^\alpha}\|v\|_{C^\beta}.
\end{equation}
If moreoever $\alpha+\beta>0$, we have
\begin{equation}
\|\PI(u,v)\|_{C^{\alpha+\beta}}\lesssim\|u\|_{C^\alpha}\|v\|_{C^\beta}.
\end{equation}
In Sobolev spaces, we have for $\alpha>0$
\begin{equation}
\|P_uv\|_{H^{\beta}}\lesssim\|u\|_{H^\alpha}\|v\|_{C^\beta}\quad\text{and}\quad\|P_uv\|_{H^{\beta}}\lesssim\|u\|_{C^\alpha}\|v\|_{H^\beta}
\end{equation}
while for $\alpha<0$ we get
\begin{equation}
\|P_uv\|_{H^{\alpha+\beta}}\lesssim\|u\|_{H^\alpha}\|v\|_{C^\beta}\quad\text{and}\quad\|P_uv\|_{H^{\alpha+\beta}}\lesssim\|u\|_{C^\alpha}\|v\|_{H^\beta}
\end{equation}
and finally if $\alpha+\beta>0$,
\begin{equation}
\|\PI(u,v)\|_{H^{\alpha+\beta}}\lesssim\|u\|_{H^\alpha}\|v\|_{C^\beta}.
\end{equation}
\end{proposition}

\medskip

Finally, we introduce the truncated paraproduct
\begin{equation}
\P_u^Nv:=\sum_{\substack{n<m-1\\2^n,2^m\ge N}}\Delta_nu\Delta_mv
\end{equation}
for any $N\ge1$. Since $\P^N-\P$ only depends on a finite number of modes, it is a smoothing operator while $\P^N$ satisfies the same continuity result as $\P$ for fixed $N$.

\medskip

\begin{proposition}\label{PropParaproductTrunc}
Let $\alpha>0$. Then for any $N\ge1$ and $\gamma\in[0,\alpha)$, we have
\begin{equation}
\|\P_u^Nv\|_{H^\gamma(\IT)}\lesssim N^{\frac{\alpha-\beta}{2}}\|u\|_{L^2(\IT)}\|v\|_{C^\alpha(\IT)}.
\end{equation}
\end{proposition}

\vspace{2cm}

\noindent \textcolor{gray}{$\bullet$} A. Debussche -- Univ Rennes, CNRS, IRMAR - UMR 6625, F- 35000 Rennes, France.\\
{\it E-mail}: arnaud.debussche@ens-rennes.fr

\smallskip

\noindent \textcolor{gray}{$\bullet$} A. Mouzard -- CNRS \& Department of Mathematics and Applications, ENS Paris, 45 rue d'Ulm, 75005 Paris, France.\\
{\it E-mail}: antoine.mouzard@math.cnrs.fr


\begin{thebibliography}{10}

\bibitem{BahourChemin99}
{\sc H.~Bahouri and J.-Y. Chemin}, {\em Quasilinear wave equations and
  {Strichartz} estimates}, Am. J. Math., 121 (1999), pp.~1337--1377.

\bibitem{BCD}
{\sc H.~Bahouri, J.-Y. Chemin, and R.~Danchin}, {\em Fourier analysis and
  nonlinear partial differential equations}, vol.~343 of Grundlehren Math.
  Wiss., Berlin: Heidelberg, 2011.

\bibitem{BailleulDangMouzard22}
{\sc I.~Bailleul, N.~V. Dang, and A.~Mouzard}, {\em Analysis of the anderson
  operator}, 2022.

\bibitem{Bony}
{\sc J.-M. {Bony}}, {\em {Calcul symbolique et propagation des singularites
  pour les \'equations aux d\'eriv\'ees partielles non lin\'eaires.}}, {Ann.
  Sci. \'Ec. Norm. Sup\'er. (4)}, 14 (1981), pp.~209--246.

\bibitem{Bourgain93}
{\sc J.~Bourgain}, {\em Fourier transform restriction phenomena for certain
  lattice subsets and applications to nonlinear evolution equations. {I}:
  {Schr{\"o}dinger} equations}, Geom. Funct. Anal., 3 (1993), pp.~107--156.

\bibitem{Bourgain94}
\leavevmode\vrule height 2pt depth -1.6pt width 23pt, {\em Periodic nonlinear
  {Schr{\"o}dinger} equation and invariant measures}, Commun. Math. Phys., 166
  (1994), pp.~1--26.

\bibitem{BGT}
{\sc N.~{Burq}, P.~{G\'erard}, and N.~{Tzvetkov}}, {\em {Strichartz
  inequalities and the nonlinear Schr\"odinger equation on compact manifolds}},
  {Am. J. Math.}, 126 (2004), pp.~569--605.

\bibitem{CoifmanMeyer}
{\sc R.~R. {Coifman} and Y.~{Meyer}}, {\em {Au dela des op\'erateurs
  pseudo-diff\'erentiels}}, vol.~57, Soci\'et\'e Math\'ematique de France
  (SMF), Paris, 1978.

\bibitem{DaPrato}
{\sc G.~Da~Prato}, {\em An introduction to infinite-dimensional analysis},
  Universitext, Berlin: Springer, 2006.

\bibitem{DaPratoDebussche03}
{\sc G.~Da~Prato and A.~Debussche}, {\em Strong solutions to the stochastic
  quantization equations.}, Ann. Probab., 31 (2003), pp.~1900--1916.

\bibitem{DaPratoZabczyk2014}
{\sc G.~Da~Prato and J.~Zabczyk}, {\em Stochastic equations in infinite
  dimensions}, vol.~152 of Encycl. Math. Appl., Cambridge: Cambridge University
  Press, 2nd ed.~ed., 2014.

\bibitem{DRTV24}
{\sc A.~Debussche, R.~Liu, N.~Tzvetkov, and N.~Visciglia}, {\em Global
  well-posedness of the 2d nonlinear {Schr{\"o}dinger} equation with
  multiplicative spatial white noise on the full space}, Probab. Theory Relat.
  Fields, 189 (2024), pp.~1161--1218.

\bibitem{DW}
{\sc A.~Debussche and H.~Weber}, {\em The {S}chr\"{o}dinger equation with
  spatial white noise potential}, Electron. J. Probab., 23 (2018), pp.~Paper
  No. 28, 16.

\bibitem{DumazLabbe}
{\sc L.~Dumaz and C.~Labb\'{e}}, {\em Localization of the continuous {A}nderson
  {H}amiltonian in 1-{D}}, Probab. Theory Related Fields, 176 (2020),
  pp.~353--419.

\bibitem{FN}
{\sc M.~Fukushima and S.~Nakao}, {\em On spectra of the {S}chr\"{o}dinger
  operator with a white {G}aussian noise potential}, Z.
  Wahrscheinlichkeitstheorie und Verw. Gebiete, 37 (1976/77), pp.~267--274.

\bibitem{GUZ}
{\sc M.~Gubinelli, B.~Ugurcan, and I.~Zachhuber}, {\em Semilinear evolution
  equations for the {A}nderson {H}amiltonian in two and three dimensions},
  Stoch. Partial Differ. Equ. Anal. Comput., 8 (2020), pp.~82--149.

\bibitem{HairerLabbe15}
{\sc M.~Hairer and C.~Labb{\'e}}, {\em A simple construction of the continuum
  parabolic {Anderson} model on {{\(\mathbf{R}^2\)}}}, Electron. Commun.
  Probab., 20 (2015), p.~11.
\newblock Id/No 43.

\bibitem{LRS}
{\sc J.~L. Lebowitz, H.~A. Rose, and E.~R. Speer}, {\em Statistical mechanics
  of the nonlinear {Schr{\"o}dinger} equation.}, J. Stat. Phys., 50 (1988),
  pp.~657--687.

\bibitem{MatsudaZuijlen22}
{\sc T.~Matsuda and W.~van Zuijlen}, {\em Anderson hamiltonians with singular
  potentials}, 2022.

\bibitem{Mouzard}
{\sc A.~Mouzard}, {\em Weyl law for the {Anderson} {Hamiltonian} on a
  two-dimensional manifold}, Ann. Inst. Henri Poincar{\'e}, Probab. Stat., 58
  (2022), pp.~1385--1425.

\bibitem{MouzardOuhabaz23}
{\sc A.~Mouzard and E.~M. Ouhabaz}, {\em A simple construction of the anderson
  operator via its quadratic form in dimensions two and three}, 2309.02821,
  (2023).

\bibitem{MZ}
{\sc A.~Mouzard and I.~Zachhuber}, {\em Strichartz inequalities with white
  noise potential on compact surfaces}, Anal. PDE, 17 (2024), pp.~421--454.

\bibitem{OhSosoeTolomeo22}
{\sc T.~Oh, P.~Sosoe, and L.~Tolomeo}, {\em Optimal integrability threshold for
  {Gibbs} measures associated with focusing {NLS} on the torus}, Invent. Math.,
  227 (2022), pp.~1323--1429.

\bibitem{Ouhabaz05}
{\sc E.~M. Ouhabaz}, {\em Analysis of heat equations on domains}, vol.~31 of
  Lond. Math. Soc. Monogr. Ser., Princeton, NJ: Princeton University Press,
  2005.

\bibitem{PaleyZygmund30}
{\sc R.~E. A.~C. Paley and A.~Zygmund}, {\em On some series of functions. {I},
  {II}.}, Proc. Camb. Philos. Soc., 26 (1930), pp.~337--357, 458--474.

\bibitem{PaleyZygmund32}
{\sc R.~E. A.~C. Paley and A.~Zygmund}, {\em On some series of functions.
  {III}}, Proc. Camb. Philos. Soc., 28 (1932), pp.~190--205.

\bibitem{Strook1993}
{\sc D.~W. Stroock}, {\em Probability theory: an analytic view}, Cambridge:
  Cambridge University Press, 1993.

\bibitem{Tataru00}
{\sc D.~Tataru}, {\em Strichartz estimates for operators with nonsmooth
  coefficients and the nonlinear wave equation}, Am. J. Math., 122 (2000),
  pp.~349--376.

\bibitem{TzvetkovVisciglia23}
{\sc N.~Tzvetkov and N.~Visciglia}, {\em Global dynamics of the {{\(2d\)}}
  {NLS} with white noise potential and generic polynomial nonlinearity},
  Commun. Math. Phys., 401 (2023), pp.~3109--3121.

\bibitem{TzvetkovVisciglia23bis}
\leavevmode\vrule height 2pt depth -1.6pt width 23pt, {\em Two dimensional
  nonlinear {Schr{\"o}dinger} equation with spatial white noise potential and
  fourth order nonlinearity}, Stoch. Partial Differ. Equ., Anal. Comput., 11
  (2023), pp.~948--987.

\end{thebibliography}
\end{document}